\documentclass[10pt]{amsart}
\usepackage{amssymb}
\usepackage{bm}
\usepackage{graphicx}
\usepackage[centertags]{amsmath}
\usepackage{amsfonts}
\usepackage{amsthm}
\newtheorem{thm}{Theorem}
\newtheorem{cor}[thm]{Corollary}
\newtheorem{lem}[thm]{Lemma}
\newtheorem{prop}[thm]{Proposition}
\newtheorem{defn}[thm]{Definition}
\newtheorem{claim}{Claim}
\newtheorem{fact}{Fact}
\theoremstyle{definition}
\newtheorem{rem}{Remark}
\newtheorem{examp}{Example}

\newcommand{\SB}{\mathbf{\Sigma}}
\newcommand{\PB}{\mathbf{\Pi}}

\newcommand{\nn}{\mathbb{N}}
\newcommand{\rr}{\mathbb{R}}
\newcommand{\aaa}{\mathcal{A}}
\newcommand{\bbb}{\mathcal{B}}
\newcommand{\ccc}{\mathcal{C}}
\newcommand{\eee}{\mathcal{E}}

\newcommand{\iii}{\mathcal{I}}
\newcommand{\jjj}{\mathcal{J}}
\newcommand{\fff}{\mathcal{F}}
\newcommand{\kkk}{\mathcal{K}}
\newcommand{\hhh}{\mathcal{H}}

\newcommand{\wf}{\mathrm{WF}}

\newcommand{\sssb}{\Sigma_{\mathbf{b}}}
\newcommand{\pap}{\mathcal{P}_{\infty}}
\newcommand{\bt}{\nn^{<\nn}}

\newcommand{\ct}{2^{<\nn}}
\newcommand{\sg}{\sigma}

\newcommand{\con}{\smallfrown}

\newcommand{\llf}{\mathcal{L}_{f}}

\begin{document}


\title[On pairs of definable orthogonal families]{On pairs of definable orthogonal families}
\author{Pandelis Dodos and Vassilis Kanellopoulos}
\address{National Technical University of Athens, Faculty of Applied Sciences,
Department of Mathematics, Zografou Campus, 157 80, Athens, Greece}
\email{pdodos@math.ntua.gr, bkanel@math.ntua.gr}


\maketitle


\begin{abstract}
We introduce the notion of an M-family of infinite subsets of $\nn$
which is implicitly contained in the work of A. R. D. Mathias. We
study the structure of a pair of orthogonal hereditary families
$\aaa$ and $\bbb$, where $\aaa$ is analytic and $\bbb$ is
$C$-measurable and an M-family.
\end{abstract}


\section{Introduction}

Two families $\aaa$ and $\bbb$ of infinite subsets of $\nn$
are said to be \textit{orthogonal} if $A\cap B$ is finite
for every $A\in\aaa$ and every $B\in\bbb$. The study of the
structure of a pair $(\aaa, \bbb)$ of orthogonal families is
a classical topic (\cite{Hau}) which has found numerous
applications (see, for  instance, \cite{DW} and \cite{To4}).
Among all pairs $(\aaa, \bbb)$ of orthogonal families of particular
importance is the study of the \textit{definable} ones. Here
the word definable refers to the descriptive set theoretic
complexity of $\aaa$ and $\bbb$ as subsets of $\mathcal{P}(\nn)$.
A fundamental result in this direction is the ``perfect Lusin gap"
theorem of Stevo Todor\v{c}evi\'{c} \cite{To2} which deals with
a pair of analytic and orthogonal families.

In this paper we study the structure of a pair $(\aaa, \bbb)$ of
hereditary and orthogonal families where $\aaa$ is analytic and
$\bbb$ is $C$-measurable\footnote[2]{We recall that a subset of
a Polish space is $C$-measurable if it belongs to the
smallest $\sigma$-algebra that contains the open sets and is
closed under the Souslin operation.} and ``large". Our notion
of largeness is the following which is implicitly contained
in the work of A. R. D. Mathias \cite{Ma}.
\begin{defn}
\label{d1} We say that a hereditary family $\aaa$ of infinite
subsets of $\nn$ is an M-family if for every sequence $(A_n)_n$
in $\aaa$ there exists $A\in\aaa$ whose all but finitely many
elements are in $\bigcup_{i\geq n} A_i$ for every $n\in\nn$.
\end{defn}
We should point out that there are several other notions
appearing in the literature, such as P-ideals
(see \cite{So}, \cite{To2}) or semi-selective co-ideals
(see \cite{Fa}), involving the existence of diagonal sequences.
We should also point out that the notion of an M-family
is closely related to the weak diagonal sequence property
of topological spaces and, in fact, it can be considered as
its combinatorial analogue.

Using Ellentuck's theorem \cite{El} we show that the class
of $C$-measurable M-families possesses strong stability
properties. It is closed, for instance, under intersection
and ``diagonal" products. As a consequence we prove that if
$(X,\tau_1)$ and $(Y,\tau_2)$ are two countable analytic
spaces with the weak diagonal sequence property, then the
product $(X\times Y, \tau_1\times \tau_2)$ has the weak
diagonal sequence property. This answers Question 5.4 from
\cite{TU}.

Our first result, concerning the structure of a pair $(\aaa,\bbb)$
as described above, is the following (see \S 2 for the relevant
definitions).
\medskip

\noindent \textbf{Theorem I.} \textit{Let $\aaa$ and $\bbb$ be two
hereditary, orthogonal families of infinite subsets of $\nn$.
Assume that $\aaa$ is analytic and that $\bbb$ is an M-family
and $C$-measurable. Then, either
\begin{enumerate}
\item[(i)] $\aaa$ is countably generated in $\bbb^\perp$, or
\item[(ii)] there exists a perfect Lusin gap inside $(\aaa,\bbb)$.
\end{enumerate} }
\medskip

\noindent Theorem I shows that the assumption of being an
M-family can successfully replace analyticity in the perfect
Lusin gap theorem of \cite{To2}. We should point out that the
phenomenon of replacing analyticity by a structural property
and still getting the same conclusion as in Theorem I has
already appeared in the literature (see \cite{To4} and
\cite{TU}). As a matter of fact Theorem I was motivated
by these applications.

Our second result, concerning the structure of a pair
$(\aaa,\bbb)$ as in Theorem I, extends a result
of A. Krawczyk from \cite{Kra}. To state it,
it is useful to look at the second orthogonal
$\bbb^{\perp\perp}$ of $\bbb$. In a sense the family
$\bbb^{\perp\perp}$ is the ``completion" of $\bbb$, as an
infinite subset $L$ of $\nn$ belongs to $\bbb^{\perp\perp}$
if (and only if) every infinite subset of $L$ contains an
element of $\bbb$. To proceed with our discussion, let $\ccc$
be the family of all infinite chains of $\bt$ (we recall
that a subset of $\bt$ is called a chain if it is linearly
ordered under the order of end-extension). Let also
$\iii_{\mathrm{wf}}$ be the ideal on $\bt$ generated
by the set $\mathrm{WF}$ of all downwards closed, well-founded,
infinite subtrees of $\bt$. The following theorem shows that
if $\aaa,\bbb$ are as above and $\aaa$ is not countably
generated in $\bbb^{\perp}$, then the pair
$(\ccc,\iii_{\mathrm{wf}})$ ``embeds" into the pair
$(\aaa,\bbb^{\perp\perp})$ in a very canonical way.
\medskip

\noindent \textbf{Theorem II.} \textit{Let $\aaa$ and $\bbb$ be
two hereditary, orthogonal families of infinite subsets of $\nn$.
Assume that $\aaa$ is analytic and that $\bbb$ is an M-family and
$C$-measurable. Then, either
\begin{enumerate}
\item[(i)] $\aaa$ is countably generated in $\bbb^\perp$, or
\item[(ii)] there exists a one-to-one map $\psi:\bt\to\nn$ such
that
\[ \ccc\subseteq \{ \psi^{-1}(A):A\in\aaa\} \ \text{ and } \
\iii_{\mathrm{wf}}\subseteq \{\psi^{-1}(B): B\in \bbb^{\perp\perp}\}. \]
\end{enumerate} }
\medskip

One of the main ingredients of the proofs of Theorem I
and of Theorem II is the infinite dimensional extension
of Hindman's theorem \cite{Hi}, due to K. Milliken \cite{Mil}.
It is used in a spirit similar as in \cite{ADK}.

The paper is organized as follows. In \S 2 we gather some
preliminaries needed in the rest of the paper. In \S 3 we
study the connection of M-families with other related notions
and we give some examples. In \S 4 we present some of their
structural properties. The proof of Theorem I is given in
\S 5 while the proof of Theorem II is given in \S 6.
Our general notation and terminology is standard,
as can be found in \cite{Kechris} and \cite{To3}.
\bigskip

\noindent \textbf{Acknowledgments.} We would like to
thank the anonymous referee for his thorough report
which improved the presentation of the paper.


\section{Preliminaries}

It is a common fact that once one is willing to present some results
about trees, ideals and related combinatorics, then one has to set up
a, rather large, notational system. Below we gather all the conventions
that we need and which are, more or less, standard. In what follows
$X$ will be a countable (infinite) set.

\subsection{Ideals} By $\pap(X)$ we denote the set of all infinite
subsets of $X$ (which is clearly a Polish subspace of $2^X$).
A family $\aaa\subseteq \pap(X)$ is \textit{hereditary} if for
every $A\in\aaa$ and every $A'\in \pap(A)$ we have $A'\in \aaa$.
A subfamily $\bbb$ of a family $\aaa$ is \textit{cofinal} in $\aaa$
if for every $A\in\aaa$ there exists $B\in \pap(A)$ with $B\in\bbb$.

Given $A, B\in \pap(X)$ we write $A\subseteq^* B$ if the set
$A\setminus B$ is finite, while we write $A\perp B$ if the set
$A\cap B$ is finite. Two families $\aaa, \bbb\subseteq \pap(X)$
are said to be \textit{orthogonal}, in symbols $\aaa\perp\bbb$,
if $A\perp B$ for every $A\in\aaa$ and every $B\in\bbb$.
For every $\aaa\subseteq \pap(X)$ we set $\aaa^\perp=\{
B\in\pap(X): B\perp A \text{ for all } A\in\aaa\}$ and
$\aaa^*=\{X\setminus A: A\in\aaa\}$. The family $\aaa^\perp$
is called the \textit{orthogonal} of $\aaa$. Notice that
$\aaa^\perp$ is an ideal.

Two families $\aaa$ and $\bbb$ are \textit{countably separated} if
there exists a sequence $(C_n)_n$ in $\pap(X)$ such that
for every $A\in\aaa$ and every $B\in\bbb$ there exists $n\in\nn$
with $A\subseteq C_n$ and $C_n \perp B$. A family $\aaa$ is
\textit{countably generated} in a family $\bbb$, if there
exists a sequence $(B_n)_n$ in $\bbb$ such that for every $A\in\aaa$
there exists $n\in\nn$ with $A\subseteq^* B_n$. An ideal $\iii$
on $X$ is said to be \textit{bisequential} if for every ultrafilter
$p$ on $X$ with $\iii\subseteq p^*$ the family $\iii$ is countably
generated in $p^*$.

Given $\aaa\subseteq \pap(X)$ we let
\begin{equation}
\label{e1} \mathrm{co}(\aaa)=\{ B\in\pap(X): \exists A\in\aaa \text{ with }
B\cap A \text{ infinite}\}= \pap(X)\setminus \aaa^\perp.
\end{equation}
Notice that $\mathrm{co}(\aaa)$ is a co-ideal. We call
$\mathrm{co}(\aaa)$ as the \textit{co-ideal generated by}
$\aaa$. Observe that if $\aaa$ is hereditary, then
$\mathrm{co}(\aaa)=\{B\in\pap(X): \exists A\in\aaa \text{ with }
A\subseteq B\}$.

The following elementary, well-known, fact provides the description
of the second orthogonal $\aaa^{\perp\perp}$ of a hereditary family $\aaa$.
\begin{fact}
\label{f1} Let $\aaa\subseteq \pap(X)$ hereditary. Let also $B\in\pap(X)$.
Then $B\in\aaa^{\perp\perp}$ if and only if for every $C\in\pap(B)$
there exists $A\in\pap(C)$ with $A\in\aaa$.
\end{fact}
An ideal $\iii$ is said to have the \textit{Fr\'{e}chet property}
if $\iii=\iii^{\perp\perp}$. We notice that if $\aaa$ is a
hereditary family, then both $\aaa^\perp$ and $\aaa^{\perp\perp}$
have the Fr\'{e}chet property. The following fact is also well-known.
We sketch its proof for completeness.
\begin{fact}
\label{f2} A bisequential ideal $\iii$ on $X$ has the Fr\'{e}chet property.
\end{fact}
\begin{proof}
In light of Fact \ref{f1}, it is enough to show that for every
$A\notin \iii$ there exists $C\in \pap(A)$ with $C\in \iii^\perp$.
So, let $A\notin \iii$. The family $\{A\setminus L:L\in\iii\}$ has the
finite intersection property. Hence, we may find $p\in\beta X$,
non-principal, with $\iii\subseteq p^*$ and $A\in p$.
By the bisequentiality of $\iii$, there exists a sequence $(B_n)_n$
in $p^*$ such that for every $L\in\iii$ there exists $n\in\nn$ with
$L\subseteq^* B_n$. Clearly, we may assume that the sequence
$(B_n)_n$ is increasing. Let $C$ be an infinite diagonalization of
the decreasing sequence $(A\setminus B_n)_n$. Then $C\in\pap(A)$
and $C\in \iii^\perp$. The proof is completed.
\end{proof}

\subsection{Trees and block sequences} By $X^{<\nn}$ we shall
denote the set of all finite sequences in $X$. We view $X^{<\nn}$
as a tree under the (strict) partial order $\sqsubset$ of
end-extension. For every $s,t\in X^{<\nn}$
by $s^{\con}t$ we denote their concatenation. If $T$ is a
downwards closed subtree of $X^{<\nn}$, then by $[T]$ we shall
denote its body, i.e. the set $\{\sg\in X^\nn:\sg|n\in T \
\forall n\in\nn\}$. Two nodes $s,t\in T$ are said to be
\textit{comparable} if either $t\sqsubseteq s$ or
$s\sqsubseteq t$; otherwise they are said to be \textit{incomparable}.
A subset of $T$ consisting of pairwise comparable nodes
is said to be a \textit{chain}, while a subset of $T$ consisting
of pairwise incomparable nodes is said to be an
\textit{antichain}.

By $\Sigma$ we shall denote the downwards closed subtree of
$\bt$ consisting of all strictly increasing finite sequences.
We view, however, every $t\in\Sigma$ not only as a finite
increasing sequence but also as finite subset
of $\nn$. Given $s,t\in\Sigma\setminus\{\varnothing\}$ we write
$s<t$ if $\max s<\min t$. By convention $\varnothing<t$ for every
$t\in\Sigma$ with $t\neq \varnothing$. If $s,t\in\Sigma$ with $s<t$,
then we will frequently denote by $s\cup t$ the concatenation
of $s$ and $t$.

By $\mathbf{B}$ we shall denote the closed subset of $\Sigma^\nn$
($\Sigma$ equipped with the discrete topology) consisting
of all sequences $(b_n)_n$ with $b_n\neq\varnothing$ and
$b_n<b_{n+1}$ for every $n\in\nn$. We call a sequence
$\mathbf{b}=(b_n)_n\in\mathbf{B}$ a \textit{block} sequence.
For every block sequence $\mathbf{b}=(b_n)_n$ we set
\begin{equation}
\label{e2} \langle \mathbf{b}\rangle=\big\{ \bigcup_{n\in F} b_n: F\subseteq\nn
\text{ finite}\big\}\subseteq\Sigma \ \text{ and } \ [\mathbf{b}]=\big\{
(c_n)_n\in\mathbf{B}: c_n\in \langle\mathbf{b}\rangle \ \forall n\big\}.
\end{equation}
We will need the following consequence of K. Milliken's theorem \cite{Mil}.
\begin{thm}
\label{t2} Let $\mathcal{X}$ be a $C$-measurable subset of
$\mathbf{B}$. Then there exists $\mathbf{b}\in \mathbf{B}$ such that
either $[\mathbf{b}]\subseteq \mathcal{X}$ or $\mathcal{X}\cap [\mathbf{b}]=
\varnothing$.
\end{thm}
We recall that the class of $C$-measurable sets is strictly bigger
than the $\sg$-algebra generated by the analytic sets (see, for
instance, \cite{Kechris}).

\subsection{Lusin gaps and related results} Let $\aaa, \bbb\subseteq
\pap(X)$. A \textit{perfect Lusin gap} inside $(\aaa,\bbb)$ is a
continuous, one-to-one map $2^\nn\ni x \mapsto (A_x, B_x)\in \aaa\times \bbb$
such that the following are satisfied.
\begin{enumerate}
\item[(a)] For every $x\in 2^\nn$, $A_x\cap B_x=\varnothing$.
\item[(b)] For every $x, y\in 2^\nn$ with $x\neq y$,
$(A_x\cap B_y)\cup (A_y\cap B_x)\neq \varnothing$.
\end{enumerate}
The notion of a perfect Lusin gap was introduced by S.
Todor\v{c}evi\'{c}. We notice that if there exists a perfect Lusin
gap inside $(\aaa, \bbb)$, then $\aaa$ and $\bbb$ are not
countably separated. The following result of Todor\v{c}evi\'{c}
\cite{To2} shows that this the only case for a pair of analytic
and orthogonal families.
\begin{thm}
\label{t3} Let $\aaa$ and $\bbb$ be two analytic, hereditary and
orthogonal families of infinite subsets of $\nn$. Then, either
\begin{enumerate}
\item[(i)] $\aaa$ and $\bbb$ are countably separated, or
\item[(ii)] there exists a perfect Lusin gap inside
$(\aaa, \bbb)$.
\end{enumerate}
\end{thm}
Theorem \ref{t3} is a consequence of the Open Coloring Axiom
for $\SB^1_1$ sets (see \cite{F}, \cite{To1}). We should
point out that it is the perfectness of the gap which is
essential in many applications. We refer the reader to
\cite{To2} and \cite{To4} for more information.

We will also need the following slight reformulation of
\cite[Theorem 3]{To2}.
\begin{thm}
\label{t4} Let $\aaa, \bbb\subseteq \pap(\nn)$ be two hereditary
orthogonal families. Assume that $\aaa$ is analytic and not
countably generated in $\bbb^\perp$. Then there exists a
one-to-one map $\phi:\Sigma\to\nn$ such that, setting
\[ \eee=\{\phi^{-1}(A): A\in\aaa\} \ \text{ and } \
\hhh=\{\phi^{-1}(B): B\in\bbb\},\]
the following are satisfied.
\begin{enumerate}
\item[(i)] For every $\sg\in [\Sigma]$ the set $\{\sg|n:n\in\nn\}$ belongs to $\eee$.
\item[(ii)] For every $t\in\Sigma$ the set $\big\{t\cup \{n\}: n\in\nn \text{ and }
t<\{n\}\big\}$ of immediate successors of $t$ in $\Sigma$ belongs to $\hhh$.
\end{enumerate}
\end{thm}
\begin{proof}
Assume that $\aaa$ is analytic, hereditary and not countably generated in
$\bbb^\perp$. By \cite[Theorem 3]{To2}, there exists a downwards closed
subtree $T$ of $\Sigma$ such that the following are satisfied.
\begin{enumerate}
\item[(B1)] For every $\sg\in [T]$, $\{\sg(n):n\in\nn\}\in\aaa$.
\item[(B2)] For every $t\in T$, the set $\{n\in\nn: t<\{n\} \text{ and }
t\cup \{n\}\in T\}$ is infinite and is included in an element of $\bbb$.
\end{enumerate}
Recursively and using property (B2) above, we may select a
downwards closed subtree $S$ of $T$ such that the following hold.
\begin{enumerate}
\item[(a)] For all $s\in S$, the set $\{n\in\nn: s<\{n\} \text{
and } s\cup \{n\}\in S\}$ is infinite.
\item[(b)] For all $s, w\in S\setminus\{\varnothing\}$ with $s\neq w$,
we have $\max s\neq \max w$.
\end{enumerate}
Fix $m\in\nn$ such that $(m)\in S$ and let $S_m=\{ t\in\Sigma:
(m)^{\con}t\in S\}$. By (a) above, $S_m$ is an infinitely
splitting, downwards closed subtree of $\Sigma$. Hence,
there exists a bijection $h:\Sigma\to S_m$ such that
$|t|=|h(t)|$ for all $t\in\Sigma$ and moreover $s\sqsubset t$ if
and only if $h(s)\sqsubset h(t)$ for all $s,t\in\Sigma$. Now
define $\phi:\Sigma\to\nn$ as follows. We set
$\phi(\varnothing)=m$. For every $t\in\Sigma$ with
$t\neq \varnothing$ we set $\phi(t)=\max h(t)$. Notice that, by
(b) above, the map $\phi$ is one-to-one. It is easy to check that
$\phi$ is as desired.
\end{proof}


\section{Connections with related notions and examples}

In this section we present the relation between M-families and
other notions already studied in the literature.
Let us start with the following fact which provides characterizations
of M-families. The proof is left to the interested reader.
\begin{fact}
\label{f3} Let $X$ be a countable set and $\aaa\subseteq\pap(X)$
be a hereditary family. Then the following are equivalent.
\begin{enumerate}
\item[(i)] The family $\aaa$ is an M-family.
\item[(ii)] For every decreasing sequence $(D_n)_n$ in
$\mathrm{co}(\aaa)$ there exists $A\in\aaa$ with $A\subseteq^* D_n$
for every $n\in\nn$.
\item[(iii)] For every sequence $(A_n)_n$ in $\aaa$ there exists $A\in\aaa$
such that $A\cap A_n\neq \varnothing$ for infinitely many $n\in\nn$.
\end{enumerate}
\end{fact}
The notion of an M-family is closely related to the notion
of a selective co-ideal due to A. R. D. Mathias. We recall that a co-ideal
$\fff$ on $\nn$ is said to be \textit{selective}, or a \textit{happy family}
as it is called in \cite{Ma}, if for every decreasing sequence
$(D_n)_n$ in $\fff$ there exists $D\in\fff$ such that
$D\setminus \{0,...,n\}\subseteq D_n$ for all $n\in D$.
We have the following characterization of M-families
which justifies our terminology.
\begin{prop}
\label{p5} Let $\aaa$ be a hereditary family on $\nn$.
Then $\aaa$ is an M-family if and only if the co-ideal
$\mathrm{co}(\aaa)$ generated by $\aaa$ is selective.
\end{prop}
\begin{proof}
First assume that the co-ideal $\mathrm{co}(\aaa)$ is selective.
Let $(D_n)_n$ be a decreasing sequence in $\mathrm{co}(\aaa)$.
By the selectivity of $\mathrm{co}(\aaa)$, there exists
$D\in\mathrm{co}(\aaa)$ with $D\setminus \{0,...,n\}
\subseteq D_n$ for all $n\in D$. Pick $A\in\aaa$ with
$A\subseteq D$. Then $A\subseteq^* D_n$ for all $n\in\nn$.
By Fact \ref{f3}(ii), we see that $\aaa$ is an M-family.

Conversely, assume that $\aaa$ is an M-family. Let $(D_n)_n$
be a decreasing sequence in $\mathrm{co}(\aaa)$. By Fact
\ref{f3}(ii), there exists $A\in\aaa$ with $A\subseteq^* D_n$
for all $n\in\nn$. Recursively, we select a strictly increasing
sequence $(m_n)_n$ in $\nn$ such that $m_0=\min A$ and
$m_{n+1}\in A\cap D_{m_n}$ for every $n\in\nn$. We set
$D=\{m_n:n\in\nn\}$. Then $D\subseteq A$ and
$D\setminus \{0,...,n\}\subseteq D_n$ for all $n\in D$.
As $\aaa$ is hereditary we get that $D\in\aaa\subseteq
\mathrm{co}(\aaa)$. Hence, $\mathrm{co}(\aaa)$ is selective
and the proof is completed.
\end{proof}
The following proposition shows that the notion of an
M-family is, in a sense, the ``dual" notion of
bisequentiality.
\begin{prop}
\label{p6} Let $X$ be a countable set.
\begin{enumerate}
\item[(i)] Let $\aaa\subseteq \pap(X)$ be a hereditary family. If
$\aaa^\perp$ is bisequential, then $\aaa$ is an M-family.
\item[(ii)] Let $\iii$ be an ideal on $X$. If $\iii$ is
bisequential, then $\iii^{\perp}$ is an M-family.
\end{enumerate}
\end{prop}
\begin{proof}
(i) By Fact \ref{f3}(ii), it is enough to show that
for every decreasing sequence $(D_n)_n$ in $\mathrm{co}(\aaa)$
there exists $A\in\aaa$ with $A\subseteq^* D_n$ for every $n\in\nn$.
So, let $(D_n)_n$ be one. As $\aaa^\perp$ is an ideal, the family
$\{D_n\setminus L:n\in\nn \text{ and } L\in\aaa^\perp\}$ has the
finite intersection property. Hence, we may select $p\in \beta X$
with $\aaa^\perp\subseteq p^*$ and $D_n\in p$ for all $n\in\nn$.
Notice that $p$ is non-principal. By the bisequentiality of
$\aaa^\perp$, there exists a sequence $(C_n)_n$
in $p^*$ such that for every $B\in \aaa^\perp$ there exists
$n\in\nn$ with $B\subseteq^* C_n$. We may assume that the sequence
$(C_n)_n$ is increasing. Let $Q\in \pap(X)$ be a diagonalization
of the decreasing sequence $(D_n\setminus C_n)_n$. Then
$Q\subseteq ^* D_n$ and $Q\perp C_n$ for all $n\in\nn$. By the
properties of the sequence $(C_n)_n$ we see that $Q\notin
\aaa^\perp$. As $\aaa$ is hereditary, there exists $A\subseteq Q$
with $A\in\aaa$. Hence $A\subseteq^* D_n$ for all $n\in\nn$. Thus
$\aaa$ is an M-family.\\
(ii) By Fact \ref{f2}, the ideal $\iii$ has the Fr\'{e}chet
property. Thus $\iii^{\perp\perp}$ is bisequential and so the
result follows by part (i).
\end{proof}
We notice that the converse of Proposition \ref{p6}(i) is also
true, provided that the orthogonal $\aaa^\perp$ of $\aaa$ is
analytic. Indeed, let $\aaa$ be an M-family such that
$\aaa^\perp$ is $\SB^1_1$. By Proposition \ref{p5}, we see that
the co-ideal $\mathrm{co}(\aaa)$ generated by $\aaa$ is selective.
If follows that $\aaa^\perp$ is an analytic ideal whose complement,
$\mathrm{co}(\aaa)$, is selective. By \cite[Exercise 12.3]{To3},
we get that $\aaa^\perp$ is bisequential.

We proceed our discussion by presenting some examples of M-families.
\begin{examp}
\label{ex1} Let $\iii_{\mathrm{c}}$ be the ideal on $\bt$ generated
by the infinite chains of $\bt$. That is
\begin{equation}
\label{e3} \iii_{\mathrm{c}}=\big\{ C\in \pap(\bt): \exists
\sg_0, ..., \sg_k \in \nn^\nn \text{ with } C \subseteq
\bigcup_{i=0}^k\{\sg_i|n:n\in\nn\} \big\}.
\end{equation}
Notice that $\iii_{\mathrm{c}}$ has the Fr\'{e}chet property.
We set $\aaa= \iii_{\mathrm{c}}^\perp$. Namely
$\aaa$ consists of all infinite subsets of $\bt$ not containing an
infinite chain. Then $\aaa$ is an ideal and it is easy to see that
it is $\PB^1_1$-complete. The family $\aaa$ is an M-family. We
will give a simple argument showing this. We will use Fact \ref{f3}(ii).
So, let $(D_n)_n$ be a decreasing sequence in $\mathrm{co}(\aaa)$.
For every $n\in\nn$ there exists an infinite antichain $A_n$ of
$\bt$ with $A_n\subseteq D_n$. Let $A_n=(t^n_m)_m$ be an
enumeration of $A_n$. By an application of Ramsey's theorem, we
may assume that $|t^n_m|\leq |t^k_l|$ for all $n<m<k<l$. We set
\[ I=\big\{ (n<m<k<l)\in [\nn]^4: t^n_m \text{ is incomparable with }
t^k_l\big\}. \]
By Ramsey's theorem again, there exists $L\in \pap(\nn)$ such that either
$[L]^4\subseteq I$ or $[L]^4\cap I=\varnothing$. Let $L=\{l_0<l_1<...\}$
be the increasing enumeration of $L$. We claim that $[L]^4\subseteq I$.
If not, then $t^{l_0}_{l_1}$ is comparable with $t^{l_3}_{l_4}$ and
as $|t^{l_0}_{l_1}|\leq|t^{l_3}_{l_4}|$, we get that $t^{l_0}_{l_1}
\sqsubseteq t^{l_3}_{l_4}$. Similarly we get that $t^{l_0}_{l_2}
\sqsubseteq t^{l_3}_{l_4}$. But this implies that the nodes
$t^{l_0}_{l_1}$ and $t^{l_0}_{l_2}$ are comparable, contradicting
the fact that $A_{l_0}$ is an antichain. Thus $[L]^4\subseteq I$.
Now set $A=\{ t^{l_{2n}}_{l_{2n+1}}: n\in\nn\}$. Then
$A$ is an infinite antichain, and so, $A\in\aaa$. As
$A\subseteq^* D_n$ for all $n\in\nn$, this shows that $\aaa$
is an M-family.
\end{examp}
\begin{examp}
\label{ex2} We notice that if an ideal $\iii$ has the Fr\'{e}chet
property, then $\iii^{\perp}$ is not necessarily an M-family. For
instance, let $\iii_{\mathrm{d}}$ be the ideal of all dominated
subsets of $\bt$, that is
\begin{equation}
\label{e4} \iii_\mathrm{d}=\big\{ D\in \pap(\bt): \exists \sg\in \nn^\nn
\text{ such that } \forall t\in D \ \forall i<|t| \ t(i)<\sg(i)\big\}.
\end{equation}
Let also
\begin{equation}
\label{e5} \iii_{\mathrm{wf}}=\big\{ W\in\pap(\bt): \exists T\in\wf \text{ with }
W\subseteq T\big\}
\end{equation}
be the ideal on $\bt$ generated by the set $\wf$ of all downwards
closed, well-founded, infinite subtrees of $\bt$. Clearly
$\iii_{\mathrm{c}}\subseteq \iii_{\mathrm{d}}$. It is easy
to see that $\iii_{\mathrm{d}}^\perp=\iii_{\mathrm{wf}}$ and
$\iii_{\mathrm{wf}}^\perp=\iii_{\mathrm{d}}$. Hence, the ideal
$\iii_{\mathrm{d}}$ has the Fr\'{e}chet property. As in the above
example, we set $\aaa=\iii_{\mathrm{d}}^\perp=\iii_{\mathrm{wf}}$.
Again we see that $\aaa$ is a $\PB^1_1$-complete ideal. However,
$\aaa$ is not an M-family. To see this, for every $n\in\nn$ let
$D_n=\{t\in\bt: 0^{n+1}\sqsubseteq t\}$. Then $(D_n)_n$ is a decreasing
sequence of sets in $\mathrm{co}(\aaa)$. It is easy to check that
if $A$ is any infinite subset of $\bt$ with $A\subseteq^* D_n$
for all $n\in\nn$, then $A$ must belong to $\iii_{\mathrm{d}}$.
\end{examp}
\begin{examp}
\label{ex3} Let $E$ be a Polish space and $\mathbf{f}=\{f_n\}_n$ be
a pointwise bounded sequence of real-valued Baire-1 functions on
$E$. Assume that the closure $\kkk$ of $\{f_n\}_n$ in $\rr^E$ is a
subset of the set of all Baire-1 functions on $E$, i.e. $\kkk$ is
a separable Rosenthal compact (see \cite{Ro}). Let $f\in\kkk$ and
set
\begin{equation}
\label{e6} \llf=\big\{ L\in \pap(\nn): (f_n)_{n\in L}
\text{ converges pointwise to } f\big\}.
\end{equation}
The family $\llf$ is a $\PB^1_1$ ideal. We also let
\begin{equation}
\label{e7} \iii_f=\big\{ L\in \pap(\nn):
f\notin\overline{ \{f_n\}}^p_{n\in L}\big\}.
\end{equation}
It is easy to see that $\iii_f$ is a $\SB^1_1$ ideal. Both $\llf$
and $\iii_f$ are well studied in the literature (see \cite{ADK},
\cite{Do}, \cite{Kra}, \cite{To3}, \cite{To4}). By a result of
J. Bourgain, D. H. Fremlin and M. Talagrand \cite{BFT}, we get
that the orthogonal $\llf^\perp$ of $\llf$ is the family $\iii_f$.
An important fact concerning the structure of $\iii_f$ is that it
is bisequential. This is due to R. Pol \cite{Pol} and it can be
also derived by the results of G. Debs in \cite{De}. Hence,
by Proposition \ref{p6}(i), we see that $\llf$ is
an M-family. Let also
\begin{equation}
\label{e8} \fff_f=\big\{  L\in \pap(\nn):
f\in\overline{ \{f_n\}}^p_{n\in L}\big\}=
\pap(\nn)\setminus \iii_f.
\end{equation}
The equality $\llf^\perp=\iii_f$ yields that the co-ideal
$\mathrm{co}(\llf)$ generated by $\llf$ is the family $\fff_f$.
By Proposition \ref{p5}, it follows that $\fff_f$ is a selective
co-ideal, a fact discovered by S. Todor\v{c}evi\'{c} in \cite{To3}.
\end{examp}


\section{Properties of M-families}

This section is devoted to the study of the structural
properties of M-families. We begin by noticing the following
fact (the proof is left to the reader).
\begin{fact}
\label{f4} Let $X$ be a countable set.
\begin{enumerate}
\item[(i)] If $\aaa\subseteq \pap(X)$ is a hereditary family
and $\bbb$ is a hereditary subfamily of $\aaa$ cofinal in $\aaa$,
then $\aaa$ is an M-family if and only if $\bbb$ is.
\item[(ii)] If $\aaa, \bbb\subseteq \pap(X)$ are two M-families,
then so is $\aaa\cup \bbb$.
\end{enumerate}
\end{fact}
Most of the properties of M-families we will establish,
are derived using an infinite-dimensional
Ramsey-type argument. To state it, we need to introduce some
pieces of notation. Let $\mathbf{C}=(C_n)_n$ be a sequence in
$\pap(\nn)$ such that $C_n\cap C_m=\varnothing$ for every
$n\neq m$. For every $n\in\nn$ let $\{ x^n_0<x^n_1< ...\}$
be the increasing enumeration of the set $C_n$. We define
$\Delta_{\mathbf{C}}:\pap(\nn)\to \pap(\nn)$ as follows.
If $L\in\pap(\nn)$ with $L=\{l_0<l_1<...\}$ its
increasing enumeration, we set
\begin{equation}
\label{e9} \Delta_{\mathbf{C}}(L)=\big\{ x^{l_{2n}}_{l_{2n+1}}:n\in\nn\big\}.
\end{equation}
Notice that the map $\Delta_{\mathbf{C}}$ is continuous.
\begin{lem}
\label{l7} Let $\aaa\subseteq \pap(\nn)$ be an M-family and
$\mathbf{C}=(C_n)_n$ be a sequence in $\aaa$ such that
$C_n\cap C_m=\varnothing$ for every $n\neq m$. Assume that
$\aaa$ is $C$-measurable. Then for every $N\in\pap(\nn)$
there exists $L\in\pap(N)$ such that $\Delta_{\mathbf{C}}(M)
\in\aaa$ for all $M\in\pap(L)$.
\end{lem}
\begin{proof}
Let
\[\mathcal{C}_{\aaa}=\{ M\in\pap(\nn): \Delta_{\mathbf{C}}(M)\in\aaa\}.\]
Then $\ccc_{\aaa}$ is $C$-measurable. By Ellentuck's theorem
\cite{El}, we find $L\in\pap(N)$ such that either
$\pap(L)\subseteq \ccc_\aaa$ or $\pap(L)\cap
\ccc_\aaa=\varnothing$. It is enough to show that $\pap(L)\cap
\ccc_\aaa\neq\varnothing$. To this end we argue as follows. For
every $n\in L$ we set
\[ H_n=\{ x^n_i:i\in L \text{ and } i>n\}.\]
Then $H_n\subseteq C_n$ and so $H_n\in\aaa$ for all $n\in L$. By
Fact \ref{f3}(iii), there exists $A\in\aaa$ such that $A\cap H_n\neq
\varnothing$ for infinitely many $n\in L$. We can easily select
$M=\{m_0<m_1<...\} \in\pap(L)$ such that $x^{m_{2n}}_{m_{2n+1}}\in
A\cap H_{m_{2n}}$ for all $n\in\nn$. Then
$\Delta_{\mathbf{C}}(M)\subseteq A$. As $\aaa$ is hereditary, we
see that $\Delta_{\mathbf{C}}(M)\in \aaa$. Hence $\pap(L)\cap
\ccc_\aaa\neq\varnothing$ and the proof is completed.
\end{proof}
The following proposition is the first application of Lemma \ref{l7}.
\begin{prop}
\label{p8} Let $X$ be a countable set and $\aaa, \bbb\subseteq
\pap(X)$ be two M-families. If $\aaa$ and $\bbb$ are
$C$-measurable, then $\aaa\cap \bbb$ is an M-family.
\end{prop}
\begin{proof}
Clearly we may assume that $X=\nn$. In order to show that
$\aaa\cap \bbb$ is an M-family we will use Fact \ref{f3}(ii).
So, let $(D_n)_n$ be a decreasing sequence in
$\mathrm{co}(\aaa\cap\bbb)$. As the family
$\aaa\cap\bbb$ is hereditary, there exists a sequence
$\mathbf{C}=(C_n)_n$ in $\aaa\cap\bbb$ with $C_n\subseteq D_n$ for
all $n\in\nn$. Refining if necessary, we may assume that $C_n\cap
C_m= \varnothing$ for all $n\neq m$. Applying Lemma \ref{l7}
successively two times, we find $L\in\pap(\nn)$ such that
$\Delta_{\mathbf{C}}(M)\in \aaa$ and
$\Delta_{\mathbf{C}}(M)\in\bbb$ for all $M\in\pap(L)$. Finally
observe that $\Delta_{\mathbf{C}}(M)\subseteq^* D_n$ for every
$n\in\nn$ and every $M\in\pap(L)$. The proof is completed.
\end{proof}
Let $A, B\in\pap(\nn)$ with $A=\{x_0<x_1<...\}$ and $B=\{y_0<
y_1<...\}$ their increasing enumerations. We define the
\textit{diagonal product} $A\otimes B$ of $A$ and $B$ by
\begin{equation}
\label{e10} A\otimes B=\{ (x_n,y_n): n\in\nn\} \in \pap(\nn\times\nn).
\end{equation}
If $\aaa, \bbb\subseteq \pap(\nn)$ are two hereditary families,
then we let
\begin{equation}
\label{e11} \aaa\otimes\bbb=\{ A\otimes B: A\in\aaa \text{ and } B\in\bbb\}.
\end{equation}
Notice that $\aaa\otimes\bbb$ is a hereditary subfamily of
$\pap(\nn\times\nn)$. We have the following.
\begin{prop}
\label{p9} Let $\aaa, \bbb\subseteq \pap(\nn)$ be two M-families.
If $\aaa$ and $\bbb$ are $C$-measurable, then $\aaa\otimes \bbb$
is an M-family.
\end{prop}
\begin{proof}
Let $(D_n)_n$ be a decreasing sequence in $\mathrm{co}(\aaa\otimes
\bbb)$. There exist sequences $\mathbf{A}=(A_n)_n$ and
$\mathbf{B}=(B_n)_n$ in $\aaa$ and $\bbb$ respectively such that
$A_n\otimes B_n\subseteq D_n$ for every $n\in\nn$. As the families
$\aaa$ and $\bbb$ are hereditary, we may assume that $A_n\cap
A_m=\varnothing$ and $B_n\cap B_m=\varnothing$ for all $n\neq m$.
For every $n\in\nn$ let $\{x^n_0<x^n_1<...\}$ and
$\{y^n_0<y^n_1<...\}$ be the increasing enumerations of the sets
$A_n$ and $B_n$ respectively. Applying Lemma \ref{l7} successively
two times, we find $L\in\pap(\nn)$ such that $\Delta_{\mathbf{A}}(M)\in\aaa$
and $\Delta_{\mathbf{B}}(M)\in\bbb$ for every $M\in \pap(L)$. We may
select $I=\{i_0<i_1<...\}\in \pap(L)$ such that
$x^{i_{2n}}_{i_{2n+1}}<x^{i_{2k}}_{i_{2k+1}}$ and
$y^{i_{2n}}_{i_{2n+1}}<y^{i_{2k}}_{i_{2k+1}}$ for all $n<k$. It
follows that
\[ \Delta_{\mathbf{A}}(I)\otimes \Delta_{\mathbf{B}}(I)=\big\{
(x^{i_{2n}}_{i_{2n+1}},y^{i_{2n}}_{i_{2n+1}}): n\in\nn\big\}.\]
Hence $\Delta_{\mathbf{A}}(I)\otimes \Delta_{\mathbf{B}}(I)
\subseteq^* D_n$ for every $n\in\nn$ and
$\Delta_{\mathbf{A}}(I)\otimes \Delta_{\mathbf{B}}(I)\in
\aaa\otimes\bbb$. By Fact \ref{f3}(ii), we see that
$\aaa\otimes \bbb$ is an M-family and the proof is completed.
\end{proof}
Proposition \ref{p9} has some topological implications
which we are about to describe. Let us recall, first,
some definitions. Let $(Y,\tau)$ be a (Hausdorff) topological
space. A point $y\in Y$ is said to have the
\textit{weak diagonal sequence property} if
for every doubly indexed sequence $(y^n_k)_{n,k}$ in $Y$ with
$\lim_{k} y^n_k=y$ for all $n\in\nn$, there exists $L\in\pap(\nn)$
and for every $n\in L$ a $k_n\in\nn$ such that $\lim_{n\in L}
y^n_{k_n}=y$. The space $(Y, \tau)$ has the weak diagonal sequence
property if every point $y\in Y$ has it. Using Fact \ref{f3}(iii),
it is easy to see that if $X$ is a countable set, $\tau$ is a
topology on $X$ and $x\in X$, then the point $x$ has the weak
diagonal sequence property in the space $(X,\tau)$ if and only
if the family $\ccc_x=\{ A\in\pap(X):A\stackrel{\tau}{\to} x\}$
is an M-family. The following corollary of Proposition \ref{p9}
yields a positive answer to Question 5.4 from \cite{TU}.
\begin{cor}
\label{c10} Let $X, Y$ be two countable sets and $\tau_1, \tau_2$
two analytic topologies on $X$ and $Y$ respectively. Assume that
both $(X,\tau_1)$ and $(Y,\tau_2)$ have the weak diagonal sequence
property. Then $(X\times Y, \tau_1\times\tau_2)$ has the weak
diagonal sequence property.
\end{cor}
\begin{proof}
Clearly we may assume that $X=Y=\nn$. Let $x,y\in\nn$ arbitrary.
As we have already remarked, it is enough to show that the family
\[ \ccc_{(x,y)}=\{ C\in \pap(\nn\times\nn):
C\stackrel{\tau_1\times \tau_2}{\longrightarrow} (x,y)\}\] is an
M-family. By our assumptions on $\tau_1$ and $\tau_2$ we see that
the families
\[ \ccc_x=\{ A\in\pap(\nn): A \stackrel{\tau_1}{\rightarrow} x\}
\ \text{ and } \ \ccc_y=\{ B\in \pap(\nn): B\stackrel{\tau_2}{\rightarrow} y\}\]
are both co-analytic M-families on $\nn$. It follows by
Proposition \ref{p9} that the family $\ccc_x\otimes \ccc_y$ is an
M-family. Notice that $\ccc_x\otimes \ccc_y \subseteq
\ccc_{(x,y)}$. We let
\[ \ccc^x_{(x,y)}=\big\{ C\in \ccc_{(x,y)}: C\subseteq \{x\}\times\nn\big\}
\ \text{ and } \ \ccc^y_{(x,y)}=\big\{ C\in \ccc_{(x,y)}: C\subseteq
\nn\times \{y\}\big\}.\] As $\ccc_y$ and $\ccc_x$ are M-families,
it is easy to see that so are $\ccc^x_{(x,y)}$ and
$\ccc^y_{(x,y)}$. It follows by Fact \ref{f4}(ii) that the family
\[\bbb= \ccc^x_{(x,y)} \cup \ccc^y_{(x,y)} \cup (\ccc_x\otimes \ccc_y)\]
is an M-family. Now observe that $\bbb$ is a hereditary
subfamily of $\ccc_{(x,y)}$ which is cofinal in $\ccc_{(x,y)}$.
Hence, by Fact \ref{f4}(i), we conclude that $\ccc_{(x,y)}$ is an
M-family and the proof is completed.
\end{proof}
We notice that, after a first draft of the paper, S. Todor\v{c}evi\'{c}
informed us that he was also aware of the fact that the weak
diagonal sequence property is productive within the
class of countable analytic spaces.

We proceed by presenting another application of Lemma \ref{l7}.
To this end, let us notice that, by Fact \ref{f1}, if $\aaa$
is a hereditary family, then $\aaa$ is cofinal in $\aaa^{\perp\perp}$.
Hence, by Fact \ref{f4}(i), we see that if $\aaa$ is an M-family,
then so is $\aaa^{\perp\perp}$. We have the following strengthening
of Fact \ref{f3}(iii) for the family $\aaa^{\perp\perp}$, provided that
$\aaa$ is reasonably definable.
\begin{prop}
\label{p11} Let $X$ be a countable set and $\aaa\subseteq
\pap(X)$ be an M-family and $C$-measurable. Then, for every
sequence $(A_n)_n$ in $\aaa^{\perp\perp}$ there exists $A\in
\aaa^{\perp\perp}$ such that $A\cap A_n$ is infinite for
infinitely many $n\in\nn$.
\end{prop}
\begin{proof}
Clearly we may assume that $X=\nn$. Let $(A_n)_n$ be a sequence in
$\aaa^{\perp\perp}$. By Fact \ref{f1}, we may select a sequence
$\mathbf{C}=(C_n)_n$ in $\aaa$ such that $C_n\subseteq A_n$ for
every $n\in\nn$ and $C_n\cap C_m=\varnothing$ for all $n\neq m$.
By Lemma \ref{l7}, there exists $L\in\pap(\nn)$ such that
$\Delta_{\mathbf{C}}(M)\in\aaa$ for every $M\in \pap(L)$. For
every $n\in\nn$ let $\{x_0^n<x_1^n<...\}$ be the increasing
enumeration of the set $C_n$. We set
\[ A=\bigcup_{n\in L} \{ x_i^n: i\in L \text{ and } i>n\}. \]
We claim that $A$ is the desired set. First we notice that $A\cap
C_n$ is infinite for every $n\in L$, and so, $A\cap A_n$ is
infinite for infinitely many $n\in\nn$. What remains is to show
that $A\in\aaa^{\perp\perp}$. To this end, let $B\in\pap(A)$
arbitrary. It is easy to see that either there exists $n\in L$
such that $B\cap C_n$ is infinite, or there exists $M\in \pap(L)$
such that $\Delta_{\mathbf{C}}(M)\subseteq B$. As $\aaa$ is
hereditary and $\Delta_{\mathbf{C}}(M)\in \aaa$ for every $M\in
\pap(L)$, we see that $B$ contains an element of $\aaa$. Hence, by
Fact \ref{f1}, we conclude that $A\in\aaa^{\perp\perp}$ and the
result follows.
\end{proof}
The following corollary is simply a restatement of Proposition
\ref{p11} in the topological setting.
\begin{cor}
\label{c12} Let $X$ be a countable set and $\tau$ an analytic
topology on $X$. Assume that $(X, \tau)$ is Fr\'{e}chet and has
the weak diagonal sequence property. Let $x\in X$ and set
$\mathcal{C}_x=\{ A\in\pap(X): A\stackrel{\tau}{\to}x\}$. Then for
every sequence $(A_n)_n$ is $\mathcal{C}_x$ there exists
$A\in\mathcal{C}_x$ such that $A\cap A_n$ is infinite for
infinitely many $n\in\nn$.
\end{cor}
\begin{proof}
As we have already seen in Corollary \ref{c10}, the family
$\mathcal{C}_x$ is a co-analytic M-family. Moreover, the
assumption that $(X,\tau)$ is a Fr\'{e}chet space simply reduces
to the fact that $\mathcal{C}_x^{\perp\perp}=\mathcal{C}_x$. So
the result follows by Proposition \ref{p11}.
\end{proof}
We close this section with the following result concerning the
effect of the notion of an M-family in the context of separation
of families.
\begin{prop}
\label{p13} Let $X$ be a countable set and $\aaa, \bbb \subseteq
\pap(X)$ be two hereditary families. Assume that $\bbb$ is an
M-family. Then the following are equivalent.
\begin{enumerate}
\item[(i)] $\aaa$ and $\bbb$ are countably separated.
\item[(ii)] $\aaa$ is countably generated in $\bbb^\perp$.
\end{enumerate}
\end{prop}
\begin{proof}
It is clear that (ii) implies (i). So we only have to show
the other implication. Let us fix a sequence $(C_n)_n$
in $\pap(X)$ which separates $\aaa$ from $\bbb$. For every
$F\subseteq \nn$ finite we set $C_F=\bigcap_{n\in F} C_n$.
\begin{claim}
\label{cl1} For every $A\in\aaa$ there exists $F\subseteq \nn$
finite such that $A\subseteq C_F$ and $C_F\in \bbb^\perp$.
\end{claim}
\noindent \textit{Proof of the claim.} Assume not. Thus,
there exists $A_0\in\aaa$ such that for every $F\subseteq \nn$
finite either $A_0\nsubseteq C_F$ or $C_F\notin \bbb^\perp$.
Let
\[ L=\{ n\in\nn: A_0\subseteq C_n\}. \]
We claim that $L$ is infinite. Assume not. Then $A_0\subseteq C_L$
and so, by our assumptions, we get that $C_L\notin \bbb^\perp$.
Hence, there exists $B_L\in\bbb$ with $B_L\subseteq C_L$. It
follows that for every $n\in\nn$ either $A_0\nsubseteq C_n$ (i.e.
$n\notin L$) or $B_L\subseteq C_L\subseteq C_n$. This means that
$A_0$ and $B_L$ cannot be separated by the sequence $(C_n)_n$, a
contradiction.

Now let $L=\{l_0<l_1<...\}$ be the increasing enumeration of $L$.
For every $k\in\nn$ let $D_k=C_{l_0}\cap...\cap C_{l_k}$. Clearly
$(D_k)_k$ is a decreasing sequence. By our assumptions we see that
$D_k\notin \bbb^\perp$, and so, $D_k\in \mathrm{co}(\bbb)$ for all
$k\in\nn$. As $\bbb$ is an M-family, invoking Fact \ref{f3}(ii)
we see that there exists $B_0\in\bbb$ such that $B_0\subseteq^* D_k$
for every $k\in\nn$. It follows that $B_0\subseteq^* C_n$ for all
$n\in L$. But then, for every $n\in\nn$ we have that either
$A_0\nsubseteq C_n$ or $B_0\subseteq^* C_n$. That it, the sets
$A_0$ and $B_0$ cannot be separated by the sequence $(C_n)_n$,
a contradiction again. The claim is proved.  \hfill $\lozenge$
\bigskip

\noindent By the above claim, for every $A\in\aaa$ there exists
$F_A\subseteq \nn$ finite with $C_{F_A}\in \bbb^\perp$ and
$A\subseteq C_{F_A}$. The family $\{ C_{F_A}:A\in\aaa\}$ is
clearly countable, and so, $\aaa$ is countably generated in
$\bbb^\perp$. The proof is completed.
\end{proof}


\section{Proof of Theorem I}

This section is devoted to the proof of Theorem I stated in the
introduction. So, let $\aaa,\bbb\subseteq \pap(\nn)$ be a pair
of hereditary orthogonal families such that $\aaa$ is $\SB^1_1$
and $\bbb$ is $C$-measurable and an M-family. Assume that (i)
does not hold, i.e. $\aaa$ is not countably generated in
$\bbb^\perp$. We will find a perfect Lusin gap inside
$(\aaa,\bbb)$.

By Theorem \ref{t4}, there exists a one-to-one map
$\phi:\Sigma\to\nn$ such that, setting
\[ \eee=\{\phi^{-1}(A):A\in\aaa\} \ \text{ and } \
\hhh=\{\phi^{-1}(B): B\in\bbb\}, \]
properties (i) and (ii) of Theorem \ref{t4} are satisfied for
$\eee$ and $\hhh$. In what follows, we will work inside the tree
$\Sigma$ and with the families $\eee$ and $\hhh$. Denote by $\ccc$
the family of all infinite chains of $\Sigma$. That is
\[ \ccc=\big\{ C\in \pap(\Sigma): \exists \sg\in [\Sigma] \text{ with }
C\subseteq \{\sg|n: n\in\nn\} \big\}.\]
Clearly $\ccc$ is a $\PB^0_2$ hereditary family. We notice the
following properties of the families $\eee$ and $\hhh$.
\begin{enumerate}
\item[(P1)] $\eee$ and $\hhh$ are hereditary and orthogonal.
\item[(P2)] $\eee$ is analytic and $\ccc\subseteq \eee$.
\item[(P3)] $\hhh$ is $C$-measurable and an M-family.
\item[(P4)] For every $t\in\Sigma$, $\big\{t\cup \{n\}:n\in\nn \text{ and }
t<\{n\}\big\}\in\hhh$.
\end{enumerate}
Properties (P1)-(P4) are rather straightforward consequences of the
the way the families $\eee$ and $\hhh$ are defined and  of the fact
that the map $\phi$ is one-to-one.

We are going to define a class of subsets of $\Sigma$ which will play
a decisive role in the proof of Theorem I.
\begin{defn}
\label{d14} Let $\sg\in[\Sigma]$ and
$D\in\pap(\Sigma)$. We say that $D$ descends to $\sg$, in
symbols $D\downarrow\sg$, if for every $k\in\nn$ the set $D$ is
almost included in the set $\{t\in\Sigma: \sg|k\sqsubseteq t\}$.
We call such a set $D$ a descender.
\end{defn}
We also need to introduce some notations. Let $\mathbf{B}$
be the set of all block sequences of $\Sigma$. For every
$\mathbf{b}=(b_n)_n\in\mathbf{B}$ we set
\begin{equation}
\label{e12} \Sigma_\mathbf{b}=\{ t\in \Sigma: \exists b\in \langle
\mathbf{b}\rangle \text{ with } t\sqsubseteq b\} \ \text{ and } \
\sg_{\mathbf{b}}=\bigcup_n b_n
\end{equation}
where the set $\langle \mathbf{b}\rangle$ was defined in \S 2.2.
Clearly $\Sigma_\mathbf{b}$ is a downwards closed subtree
of $\Sigma$. Notice that $\sg_{\mathbf{b}}$ is just the
leftmost branch of the tree $\Sigma_{\mathbf{b}}$. We also
observe the following.
\begin{enumerate}
\item[(O1)] The set $[\Sigma_\mathbf{b}]$ of all branches of
$\Sigma_\mathbf{b}$ is in one-to-one correspondence with the
subsequences of $\mathbf{b}=(b_n)_n$. In particular, for every
$\sg\in [\Sigma_\mathbf{b}]$ there exists a unique subsequence
$(b_{l_n})_n$ of $(b_n)_n$, which we shall denote by
$\mathbf{b}_\sg$, such that $\sg=\bigcup_n b_{l_n}$. Moreover, the map
$[\Sigma_{\mathbf{b}}]\ni\sg \mapsto \mathbf{b}_\sg\in[\mathbf{b}]$
is continuous.
\item[(O2)] If $\mathbf{c}\in [\mathbf{b}]$, then
$\Sigma_\mathbf{c}$ is a downwards closed subtree of $\Sigma_\mathbf{b}$.
\end{enumerate}
We define $\Delta:\mathbf{B}\to \pap(\Sigma)$ by
\begin{equation}
\label{e13} \Delta\big( (b_n)_n\big)=
\Big\{ b_0\cup \{\min b_2\}, ..., \bigcup_{i=0}^{3n} b_i \cup
\{\min b_{3n+2}\}, ...\Big\}.
\end{equation}
We notice the following.
\begin{enumerate}
\item[(O3)] The map $\Delta$ is continuous.
\item[(O4)] For every block sequence $\mathbf{b}=(b_n)_n$ the set
$\Delta(\mathbf{b})$ is a subset of the tree $\Sigma_{\mathbf{b}}$,
is a descender and descends to the leftmost branch
$\sg_{\mathbf{b}}=\bigcup_n b_n$
of $\Sigma_{\mathbf{b}}$.
Moreover, the sets $\{\sg_{\mathbf{b}}|n: n\in\nn\}$ and
$\Delta(\mathbf{b})$ are disjoint.
\end{enumerate}
The following lemma is a consequence of Theorem \ref{t2}
and of the fact that $\hhh$ is an M-family. It can be considered
as a parameterized version of Lemma \ref{l7}. We notice that the
arguments in its proof follow similar lines as in
\cite[Lemma 44]{ADK}.
\begin{lem}
\label{l15} There exists $\mathbf{b}\in\mathbf{B}$ such that
$\Delta(\mathbf{c})\in \hhh$ for all $\mathbf{c}\in [\mathbf{b}]$.
\end{lem}
\begin{proof}
We let
\[ \mathcal{X}=\{\mathbf{c}\in\mathbf{B}: \Delta(\mathbf{c})\in\hhh\}.\]
Then $\mathcal{X}$ is a $C$-measurable subset of $[\mathbf{B}]$.
By Theorem \ref{t2}, there exists $\mathbf{b}=(b_n)_n\in\mathbf{B}$
such that $[\mathbf{b}]$ is monochromatic. We claim that
$[\mathbf{b}]\subseteq \mathcal{X}$. To this end, we argue as follows.
For every $n\in\nn$ we set $t_n=\bigcup_{k\leq n} b_k\in\Sigma$ and
\[ A_n=\big\{ t_n\cup \{\min b_i\}: i>n+1\big\}\in \pap(\Sigma). \]
The set $A_n$ is a subset of the set $\big\{ t_n\cup \{m\}: m\in\nn
\text{ and } t_n<\{m\}\big\}$ which, by property (P4) above, belongs
to $\hhh$. As the family $\hhh$ is hereditary, we see that $A_n\in\hhh$
for all $n\in\nn$. Invoking the fact that $\hhh$ is an M-family
and Fact \ref{f3}(iii), we find $A\in\hhh$ such that
$A\cap A_n\neq\varnothing$ for infinitely many $n\in\nn$.
We may select $L=\{l_0<l_1<...\}, M=\{i_0<i_1<...\}\in \pap(\nn)$
with $l_n+1<i_n<l_{n+1}$ and such that $t_{l_n}\cup \{\min b_{i_n}\}\in
A\cap A_{l_n}$ for all $n\in\nn$. We set $s_n=t_{l_n}\cup
\{\min b_{i_n}\}$ for all $n\in\nn$. It follows that
$\{s_n:n\in\nn\}\in\hhh$, as $\{s_n:n\in\nn\}\subseteq A\in \hhh$
and $\hhh$ is hereditary.

Now we define $\mathbf{c}=(c_n)_n\in[\mathbf{b}]$ as follows. We set
$c_0= \bigcup_{k\leq l_0} b_n$ (i.e. $c_0=t_{l_0})$, $c_1=b_{l_0+1}\cup ... \cup
b_{i_0-1}$ and $c_2=b_{i_0}$. For every $n\geq 1$ we let $I_n=[i_{n-1}+1,
l_n]$ and $J_n=[l_{n}+1, i_n-1]$ and we set
\[ c_{3n}=\bigcup_{k\in I_n} b_k \ , \ c_{3n+1}=
\bigcup_{k\in J_n} b_k \ \text{ and } \ c_{3n+2}=b_{i_n}. \]
Clearly $\mathbf{c}\in [\mathbf{b}]$ and it is easy to see
that $\Delta(\mathbf{c})=\{s_n:n\in\nn\}$. Thus, $\Delta(\mathbf{c})\in\hhh$.
It follows that $[\mathbf{b}]\cap \mathcal{X}\neq\varnothing$.
Hence $[\mathbf{b}]\subseteq \mathcal{X}$ and the lemma is proved.
\end{proof}
Let $\mathbf{b}=(b_n)_n$ be the block sequence obtained by Lemma
\ref{l15}. We set
\begin{equation}
\label{e14} \fff=\big\{ A\in\pap(\Sigma):\exists (b_{l_n})_n
\text{ subsequence of } (b_n)_n \text{ with }
A\subseteq \Delta\big((b_{l_n})_n\big)\big\}.
\end{equation}
By property (P1), the family $\hhh$ is hereditary. Hence,
using the continuity of the map $\Delta$ and the fact
that $\Delta(\mathbf{c})\in\hhh$ for every $\mathbf{c}\in[\mathbf{b}]$
we see that
\begin{enumerate}
\item[(P5)] $\fff$ is a hereditary analytic subfamily of $\hhh$.
\end{enumerate}
Consider now the tree $\Sigma_{\mathbf{b}}$ corresponding
to $\mathbf{b}$ as it was defined in (\ref{e12}) above and let
$\sg\in[\Sigma_{\mathbf{b}}]$ arbitrary. By (O1), there exists
a subsequence $\mathbf{b}_\sg=(b_{l_n})_n$ of $(b_n)_n$ such that
$\sg=\bigcup_n b_{l_n}$. By (O4) and (O2), we get that
$\Delta\big((b_{l_n})_n\big)\subseteq \Sigma_{\mathbf{b}_\sg}
\subseteq \Sigma_{\mathbf{b}}$. Moreover, the set
$\Delta\big((b_{l_n})_n\big)$ descends to $\sg$ and, by definition,
belongs to the family $\fff$. Hence, summarizing, we arrive to the
the following property of $\fff$.
\begin{enumerate}
\item[(P6)] For every $\sg\in [\sssb]$ there exists $D\in\fff$
with $D\subseteq \sssb$ and $D\downarrow \sg$.
\end{enumerate}
We have the following lemma, which is essentially a consequence of
property (P6).
\begin{lem}
\label{l16} The families $\ccc$ and $\fff$ are not countably separated.
\end{lem}
\begin{proof}
Assume, towards a contradiction, that there exists a sequence $(C_k)_k$
in $\pap(\Sigma)$ such that for every $C\in\ccc$ and every $B\in\fff$
there exists $k\in\nn$ with $C\subseteq C_k$ and $C_k\perp B$. For
every $k$ let
\[ F_k=\big\{ \sg\in [\sssb]: \{\sg|n:n\in\nn\}\subseteq C_k\big\}.\]
Then each $F_k$ is a closed subset of $[\sssb]$. Moreover
$[\sssb]=\bigcup_k F_k$.

For every $t\in\sssb$ and every $k\in\nn$ there exists
$s\in\sssb$ with $t\sqsubset s$ and such that
either $V_s\cap F_k=\varnothing$ or $V_s\subseteq F_k$, where
as usual by $V_s$ we denote the clopen subset $\{\sg\in[\sssb]:
s\sqsubset \sg\}$ of $[\sssb]$. Let us say that such
a node $s$ \textit{decides} for $(t,k)$. Observe that if $s$
decides for $(t,k)$ with $V_s\subseteq F_k$, then the set
$\{w\in\sssb: s\sqsubseteq w\}$ is a subset of $C_k$.

Recursively, we select a sequence $(s_k)_k$ in $\sssb$
such that $s_0$ decides for $(\varnothing,0)$ and
$s_{k+1}$ decides for $(s_k,k+1)$ for all $k\in\nn$. Notice
that $s_k\sqsubset s_{k+1}$. Thus, setting $\tau=\bigcup_k s_k$,
we see that $\tau\in[\sssb]$. By property (P6) above,
there exists $B_0\in \fff$ with $B_0\subseteq\sssb$ and
$B_0\downarrow \tau$. Now let $m\in\nn$ with
$\{\tau|n:n\in\nn\}\subseteq C_m$. Then $\tau\in F_m$.
As $s_m\sqsubset\tau$, we see that $V_{s_m}\cap F_m\neq\varnothing$.
The node $s_m$ decides for every $m\in\nn$, and so,
$V_{s_m}\subseteq F_m$. As we have already remarked,
this implies that $\{w\in\sssb: s_m\sqsubseteq w\}\subseteq C_m$.
As $B_0$ descends to $\tau$, $B_0\subseteq \sssb$ and
$s_m\sqsubset \tau$ we get
\[ B_0\subseteq^* \{ w\in \sssb: s_m\sqsubseteq w\}\subseteq C_m.\]
Summarizing, we see that for all $m\in\nn$ either $\{\tau|n:n\in\nn\}
\nsubseteq C_m$ or $B_0\subseteq^* C_m$. That is, the sequence
$(C_k)_k$ cannot separate the sets $\{\tau|n:n\in\nn\}$ and
$B_0$ although $\{\tau|n:n\in\nn\}\in\ccc$ and $B_0\in\fff$,
a contradiction. The lemma is proved.
\end{proof}
The families $\ccc$ and $\fff$ are hereditary, analytic and orthogonal.
Thus, applying Theorem \ref{t3} to the pair $(\ccc,\fff)$
and by Lemma \ref{l16}, we get that there exists a perfect Lusin gap
inside $(\ccc,\fff)$. As $\ccc\subseteq\eee$ and $\fff\subseteq\hhh$,
we see that there exists a perfect Lusin gap $2^\nn\ni x\mapsto
(A_x, B_x)$ inside $(\eee,\hhh)$. Now recall that the map $\phi:
\Sigma\to\nn$ obtained by Theorem \ref{t4} is one-to-one. It follows
that the map $2^\nn\ni x\mapsto \big(\phi(A_x),\phi(B_x)\big)$
is a perfect Lusin gap inside $(\aaa,\bbb)$. The proof of
Theorem I is completed.
\begin{rem}
We would like to point out that one can construct the perfect
Lusin gap inside $(\eee,\hhh)$ without invoking Theorem \ref{t3}.
This can be done as follows. Let $\mathbf{b}=(b_n)_n$ be the block
sequence obtained by Lemma \ref{l15}. First we construct,
recursively, a family $(t_s)_{s\in\ct}$ in $\Sigma_\mathbf{b}$
such that the following are satisfied.
\begin{enumerate}
\item[(C1)] For every $s, s'\in\ct$ we have $s\sqsubset s'$
if and only if $t_s\sqsubset t_{s'}$.
\item[(C2)] For every $s\in\ct$ and every $\sg\in [\Sigma_\mathbf{b}]$
with $t_{s^{\con}0}\sqsubset \sg$ we have $t_{s^{\con}1}\in
\Delta(\mathbf{b}_\sg)$, where, as in (O1) above, by $\mathbf{b}_\sg$
we denote the unique subsequence $(b_{l_n})_n$ of $(b_n)_n$ such
that $\sg=\bigcup_n b_{l_n}$.
\end{enumerate}
The construction proceeds as follows. We set $t_{\varnothing}=\varnothing$.
Assume that $t_s$ has been defined for some $s\in\ct$. We select
$\tau\in \Sigma_{\mathbf{b}}$ with $t_s\sqsubset \tau$. Let
$\mathbf{b}_\tau=(b_{l_n})_n$ be the unique subsequence of $\mathbf{b}$
with $\tau=\bigcup_n b_{l_n}$. By (O4) in the proof of Theorem I,
the set $\Delta(\mathbf{b}_\tau)$ descends to $\tau$. As $t_s\sqsubset\tau$,
there exists $t_{s^{\con}1}\in \Delta(\mathbf{b}_\tau)$ with $t_s\sqsubset
t_{s^{\con}1}$. The map $[\Sigma_\mathbf{b}]\ni \sg\mapsto \Delta(\mathbf{b}_{\sg})
\in \pap(\Sigma)$ is continuous. So, we may find a node $t_{s^{\con}0}$
incomparable to $t_{s^{\con}1}$ with $t_s \sqsubset t_{s^{\con}0}\sqsubset\tau$
and such that (C2) above is satisfied.

Having completed the construction, for every $x\in 2^\nn$ let
$\sg_x= \bigcup_n t_{x|n}\in [\Sigma_{\mathbf{b}}]$ and define
\[ A_x=\{ \sg_x|n: n\in\nn\}\in\eee \ \text{ and } \
B_x=\Delta(\mathbf{b}_{\sg_x})\in\hhh.\]
The perfect Lusin gap inside $(\eee,\hhh)$ is the map
$2^\nn\ni x\mapsto (A_x, B_x)$. It is easy to check that
it is one-to-one, continuous and $A_x\cap B_x=\varnothing$
for all $x\in 2^\nn$. Finally let $x,y\in 2^\nn$ with $x\neq y$.
We may assume that $x<y$, where $<$ stands for the lexicographical
ordering of $2^\nn$. There exists $s\in\ct$ with $s^{\con}0\sqsubset x$
and $s^{\con}1\sqsubset y$. Then $t_{s^{\con}1}\in A_y$. Moreover,
we have $t_{s^{\con}0}\sqsubset \sg_x$. By (C2) above, we see that
$t_{s^{\con}1}\in \Delta(\mathbf{b}_{\sg_x})$. Thus
$A_y\cap B_x\neq\varnothing$.
\end{rem}
\begin{rem}
Let $\aaa, \bbb\subseteq \pap(\nn)$ be two hereditary,
orthogonal, analytic families and assume that $\bbb$ is an
M-family. We notice that, in this case, the dichotomy in Theorem I
can be derived directly by Theorem \ref{t3}. To see this,
observe that if $\aaa$ is not countably generated in $\bbb^\perp$,
then, by Proposition \ref{p13}, the families $\aaa$ and $\bbb$ are
not countably separated. Thus, part (ii) of Theorem \ref{t3}
yields the existence of the gap inside $(\aaa, \bbb)$.
\end{rem}
\begin{rem}
As in Example \ref{e3}, let $E$ be a Polish space and
$\mathbf{f}=\{f_n\}_n$ be a pointwise bounded sequence of real-valued
Baire-1 functions on $E$ such that the closure $\kkk$ of $\{f_n\}_n$
in $\rr^E$ is a Rosenthal compact. We set
\begin{equation}
\label{e15} \mathcal{L}_{\mathbf{f}}=\big\{ L\in\pap(\nn):
(f_n)_{n\in L} \text{ is pointwise convergent}\big\}.
\end{equation}
For every $f\in\kkk$ let also $\llf$ be as in (\ref{e6}).
In \cite[Lemmas G.9 and G.10]{To4}, S. Todor\v{c}evi\'{c}
proved that if $f$ is any point of $\kkk$, then either
\begin{enumerate}
\item[(A1)] $f$ is a $G_\delta$ point of $\kkk$, or
\item[(A2)] there exists a perfect Lusin gap in $(\mathcal{L}_{\mathbf{f}}
\setminus\llf, \llf)$.
\end{enumerate}
Let us see how Theorem I yields the above dichotomy. So, fix
a point $f\in \kkk$. First we notice that, as it was explained
in \cite[Remark 1(2)]{Do}, by Debs' theorem \cite{De} there exists
a hereditary, Borel and \textit{cofinal} subfamily $\fff$ of
$\mathcal{L}_{\mathbf{f}}$. We set $\aaa=\fff\setminus \llf$.
Then $\aaa$ is an analytic, hereditary and cofinal subfamily
of $\mathcal{L}_{\mathbf{f}}\setminus \llf$. Moreover, as we
mentioned in Example \ref{ex3}, the family $\llf$ is a
co-analytic M-family. Noticing that $\aaa$ and $\llf$
are orthogonal, by Theorem I we get that either
\begin{enumerate}
\item[(A3)] $\aaa$ is countably generated in $\llf^\perp$, or
\item[(A4)] there exists a perfect Lusin gap in $(\aaa, \llf)$.
\end{enumerate}
Clearly, we only have to check that (A3) implies (A1). Indeed, let
$(L_k)_k$ be a sequence in $\llf^\perp$ that generates $\aaa$.
Set $V_k=\kkk\setminus \overline{\{f_n\}}^p_{n\in L_k}$ and notice
that $f\in V_k$ for every $k\in\nn$. Taking into account that
$\aaa$ is cofinal in $\mathcal{L}_{\mathbf{f}}\setminus \llf$
and using the Bourgain-Fremlin-Talagrand theorem \cite{BFT}, we
see that $\{f\}=\bigcap_k V_k$; that is the point $f$ is $G_\delta$.
\end{rem}


\section{Proof of Theorem II}

This section is devoted to the proof of Theorem II. Let
$\aaa, \bbb\subseteq \pap(\nn)$ be a pair of hereditary orthogonal
families such that $\aaa$ is analytic and $\bbb$ is $C$-measurable
and an M-family. Assume that $\aaa$ is not countably generated
in $\bbb^\perp$. By Theorem \ref{t4}, there exists a one-to-one map
$\phi:\Sigma\to\nn$ such that, setting $\eee=\{\phi^{-1}(A):A\in\aaa\}$
and $\hhh=\{\phi^{-1}(B):B\in\bbb\}$, the following properties are
satisfied for $\eee$ and $\hhh$.
\begin{enumerate}
\item[(P1)] $\eee$ and $\hhh$ are hereditary and orthogonal.
\item[(P2)] $\eee$ is analytic and $\ccc\subseteq \eee$.
\item[(P3)] $\hhh$ is $C$-measurable and an M-family.
\item[(P4)] For every $t\in\Sigma$, $\big\{t\cup \{n\}:n\in\nn \text{ and }
t<\{n\}\big\}\in\hhh$.
\end{enumerate}
As in the proof of Theorem I, we shall work inside the tree $\Sigma$
and with the families $\eee$ and $\hhh$.

We introduce the following class of subsets of $\Sigma$. It will be
used in a similar manner as the class of descenders was used in
the proof of Theorem I.
\begin{defn}
\label{d17} An infinite subset $F$ of $\Sigma$ will
be called a \textit{fan} if $F$ can be enumerated as
$\{t_n:n\in\nn\}$ and there exist $s\in\Sigma$ and a strictly
increasing sequence $(m_n)_n$ in $\nn$ with $s<\{m_0\}$ and such
that $s\cup \{m_n\}\sqsubseteq t_n$ for all $n\in\nn$.
\end{defn}
The following fact is essentially well-known. We sketch a proof
for completeness.
\begin{fact}
\label{f5} Let $A\in\pap(\Sigma)$. Then either $A$ is dominated,
or $A$ contains a fan. In particular, if $T$ is a downwards
closed, well-founded, infinite subtree of $\Sigma$, then every
infinite subset $A$ of $T$ contains a fan.
\end{fact}
\begin{proof}
Fix $A\in\pap(\Sigma)$ and let $\hat{A}=\{t\in \Sigma: \exists s\in A
\text{ with } t\sqsubseteq s\}$ be the downwards closure of $A$.
It is easy to see that if $\hat{A}$ is finitely splitting,
then $A$ must be dominated while if $\hat{A}$ is not finitely
splitting, then $A$ must contain a fan.

For the second part, let $T$ be a downwards closed, well-founded,
infinite subtree of $\Sigma$ and fix $A\in \pap(T)$.
If $\hat{A}$ is finitely splitting, then by an application of
K\"{o}nig's lemma we see that $[T]\neq\varnothing$, a contradiction.
Thus $\hat{A}$ is not finitely splitting, and so, $A$ contains a fan.
\end{proof}
Notice that if $\mathbf{b}=(b_n)_n$ is a block sequence of
$\Sigma$ and $s\in\Sigma$ with $s<b_0$, then the set $\{s\cup
b_n:n\in\nn\}$ is a fan. A fan $F$ of this form will be called a
\textit{block} fan. By $\fff_{\mathrm{Block}}$ we denote the set
of all block fans of $\Sigma$. We have the following elementary
fact.
\begin{fact}
\label{f6} Every fan contains a block fan.
\end{fact}
We define $\Phi:\mathbf{B}\to \pap(\Sigma)$ by
\begin{equation}
\label{e16} \Phi\big((b_n)_n\big)= \big\{ b_0\cup b_1\cup \{\min b_2\}, ...,
b_0\cup b_{2n+1}\cup \{\min b_{2n+2}\}, ...\big\}.
\end{equation}
We observe the following.
\begin{enumerate}
\item[(O1)] The map $\Phi$ is continuous.
\item[(O2)] For every $\mathbf{b}\in\mathbf{B}$ the set
$\Phi(\mathbf{b})$ is a block fan.
\end{enumerate}
We have the following analogue of Lemma \ref{l15}.
\begin{lem}
\label{l18} There exists $\mathbf{b}\in\mathbf{B}$ such that
$\Phi(\mathbf{c})\in\hhh$ for all $\mathbf{c}\in [\mathbf{b}]$.
\end{lem}
\begin{proof}
We let
\[ \mathcal{X}=\{ \mathbf{c}\in\mathbf{B}: \Phi(\mathbf{c})\in\hhh\}.\]
Then $\mathcal{X}$ is a $C$-measurable subset of $\mathbf{B}$.
Hence, by Theorem \ref{t2}, there exists $\mathbf{b}=(b_n)_n\in
\mathbf{B}$ such that $[\mathbf{b}]$ is monochromatic. We claim
that $[\mathbf{b}]\subseteq \mathcal{X}$. Indeed, for every
$n\geq 1$ let
\[ A_n=\big\{ b_0\cup b_n\cup \{\min b_k\}: k>n\big\}\in\pap(\Sigma). \]
The set $A_n$ is a subset of the set $\big\{b_0\cup b_n\cup \{m\}:
m\in\nn \text{ and }b_n<\{m\}\big\}$ which, by property (P4), belongs
to $\hhh$. Hence, by (P1), $A_n\in\hhh$ for all $n\in\nn$. As $\hhh$
is an M-family, by Fact \ref{f3}(iii), we may select $L=\{l_0<l_1<...\}, M=\{m_0<m_1<...\}
\in \pap(\nn)$ with $1\leq l_n<m_n<l_{n+1}$ for all $n\in\nn$ and such that
\[ \big\{ b_0\cup b_{l_n}\cup \{\min b_{m_n}\}: n\in\nn\big\}\in\hhh.\]
We define $\mathbf{c}=(c_n)_n$ by $c_0=b_0$ and
$c_{2n+1}=b_{l_n}$, $c_{2n+2}=b_{m_n}$ for every $n\in\nn$. Then
$\mathbf{c}\in [\mathbf{b}]$ and $\Phi(\mathbf{c})=\{ b_0\cup
b_{l_n}\cup \{\min b_{m_n}\}: n\in\nn\}\in\hhh$. Hence,
$[\mathbf{b}]\cap\mathcal{X}\neq\varnothing$ and the result
follows.
\end{proof}
Let $\mathbf{b}=(b_n)_n$ be the block sequence obtained by Lemma
\ref{l18}. We are going to select a subset of $\Sigma$
by defining an appropriate endomorphism of $\Sigma$ (the desired
subset will be the image of this endomorphism). In particular,
we define $h:\Sigma\to\Sigma$ as follows.
\begin{enumerate}
\item[(a)] We set $h(\varnothing)=\varnothing$.
\item[(b)] If $t=(n)$ with $n\in\nn$, we set
$h\big((n)\big)=b_0\cup b_{2n+1}\cup \{\min b_{2n+2}\}$.
\item[(c)] If $t=(n_0<...< n_k)\in \Sigma$ with $k\geq 1$, we
set
\[ h(t)=b_0\cup \Big( \bigcup_{i=0}^{k-1} (b_{2n_i+1}\cup b_{2n_i+2}) \Big)
\cup b_{2n_k+1} \cup \{\min b_{2n_k+2}\}. \]
\end{enumerate}
It is easy to see that the map $h$ is well-defined and one-to-one.
We also observe the following.
\begin{enumerate}
\item[(O3)] For every $s,t\in\Sigma$ we have $s\sqsubset t$ if and only if
$h(s)\sqsubset h(t)$. Thus, if $C\in \pap(\Sigma)$, then $C$ is a chain of
$\Sigma$ if and only if $h(C)$ is.
\end{enumerate}
The following fact shows the relation between the
maps $\Phi$ and $h$.
\begin{fact}
\label{f7} Let $F$ be a block fan of $\Sigma$. Then there exists
$\mathbf{c}\in [\mathbf{b}]$ such that $h(F)=\Phi(\mathbf{c})$.
\end{fact}
\begin{proof}
Let $(u_n)_n$ be a block sequence of $\Sigma$ and $s\in\Sigma$
with $s<u_0$ and such that $F=\{ s\cup u_n:n\in\nn\}$. For every
$n\in\nn$ there exist $s_n\in\Sigma$ and $l_n\in\nn$ with
$s_n< \{l_n\}$ and $u_n=s_n\cup \{l_n\}$ (notice
that $s_n$ may be empty). We define $\mathbf{c}=(c_n)_n
\in \mathbf{B}$ as follows. We let
\[c_0=b_0 \cup \bigcup_{k\in s} (b_{2k+1}\cup b_{2k+2})\]
with the convention that $\bigcup_{k\in s} (b_{2k+1}\cup
b_{2k+2})=\varnothing$ if $s=\varnothing$. For every $n\geq 1$
we set
\[ c_{2n+1}= \Big( \bigcup_{k\in s_n} (b_{2k+1}\cup
b_{2k+2})\Big) \cup b_{2l_n+1} \text{ and } c_{2n+2}=b_{2l_n+2}.\]
It is easy to see that $\mathbf{c}\in [\mathbf{b}]$ and
$h(F)=\Phi(\mathbf{c})$, as desired.
\end{proof}
Finally, we define $\psi:\Sigma\to\nn$ by $\psi(s)=\phi\big(h(s)\big)$
for all $s\in\Sigma$. Both $\phi$ and $h$ are one-to-one, and so,
the map $\psi$ is one-to-one too. As in Example \ref{ex2}, let
$\iii_{\mathrm{wf}}$ be the ideal on $\Sigma$ generated by the set
$\wf$ of all downwards closed, well-founded, infinite subtrees of $\Sigma$.
That is $\iii_{\mathrm{wf}}= \{ W\in\pap(\Sigma): \exists T\in\wf
\text{ with } W\subseteq T\}$.
\begin{lem}
\label{l19} The following hold.
\begin{enumerate}
\item[(i)] $\ccc\subseteq \{\psi^{-1}(A):A\in\aaa\}$.
\item[(ii)] $\fff_{\mathrm{Block}}\subseteq \{\psi^{-1}(B):B\in\bbb\}$.
\item[(iii)] $\iii_{\mathrm{wf}}\subseteq \{\psi^{-1}(B):B\in\bbb^{\perp\perp}\}$.
\end{enumerate}
\end{lem}
\begin{proof}
Part (i) is an immediate consequence of property (P2) and observation (O3)
above. Part (ii) follows by Lemma \ref{l18} and Fact \ref{f7}. To
see part (iii), fix $W\in\iii_{\mathrm{wf}}$. Let $A\in\pap(W)$
arbitrary. By Facts \ref{f5} and \ref{f6}, there exists a block
fan $F$ with $F\subseteq A$. By part (ii), we see that $\psi(F)
\in\bbb$. Hence, by Fact \ref{f1}, we conclude that $\psi(W)\in
\bbb^{\perp\perp}$, as desired.
\end{proof}
The trees $\Sigma$ and $\bt$ are isomorphic, i.e. there exists a
bijection $e:\bt\to\Sigma$ with $|e(t)|=|t|$ for all $t\in\bt$ and
such that $t_1\sqsubset t_2$ in $\bt$ if and only if
$e(t_1)\sqsubset e(t_2)$. Hence, by Lemma \ref{l19}, the proof of
Theorem II is completed.
\begin{rem}
In \cite{Kra}, A. Krawczyk proved that if $\iii$ is a bisequential
analytic ideal on $\nn$, then either,
\begin{enumerate}
\item[(A1)] $\iii$ is countably generated in $\iii$, or
\item[(A2)] there exists a one-to-one map $\psi:\bt\to\nn$ such that,
setting $\jjj=\{\psi^{-1}(A):A\in\iii\}$, we have that $\ccc\subseteq
\jjj\subseteq \iii_{\mathrm{d}}$,
\end{enumerate}
where $\ccc$ denotes the set of all infinite chains of $\bt$ while
$\iii_{\mathrm{d}}$ denotes the ideal of all infinite dominated
subsets of $\bt$. Let us see how Theorem II yields the above result.
So, fix a bisequential analytic ideal $\iii$ on $\nn$. We set
$\aaa=\iii$ and $\bbb=\iii^{\perp}$. Clearly $\aaa$ and $\bbb$
are hereditary and orthogonal families. Moreover, $\aaa$ is $\SB^1_1$
while $\bbb$ is $\PB^1_1$. By Proposition \ref{p6}(ii), we see that
$\bbb$ is an M-family. By Fact \ref{f2}, the ideal $\iii$
has the Fr\'{e}chet property, and so, $\bbb^{\perp}=\iii$
and $\bbb^{\perp\perp}=\iii^{\perp}=\bbb$. Thus, applying
Theorem II, the result follows.
\end{rem}
\begin{rem}
Let $\aaa$ and $\bbb$ be as in Theorem II and assume that
$\aaa$ is not countably generated in $\bbb^{\perp}$. Let
$\psi:\bt\to\nn$ be the one-to-one map obtained by
Theorem II. Notice that for every downwards closed,
infinite subtree $T$ of $\bt$ we have that $T\in\wf$
if and only if $\psi(T)\in \bbb^{\perp\perp}$, i.e. the
set $\wf$ is Wadge reducible to $\bbb^{\perp\perp}$. Thus,
if $\aaa$ is not countably generated in $\bbb^{\perp}$, then
the family $\bbb^{\perp\perp}$ is at least $\PB^1_1$-hard.
\end{rem}


\end{document}